\documentclass[uplatex,12pt]{article}  
\usepackage{amsmath}  
\usepackage{amssymb}   
\usepackage{amsthm}
\usepackage{mathrsfs}
\newtheorem{theorem}{Theorem}  
\newtheorem{lemma}{Lemma}
\newtheorem{prop}{Proposition}

\makeatother
\title{Minimal tropical basis for Bergman fan of matroid}
\date{\today}

\author{Nakajima Yasuhito}

\begin{document}  
\maketitle
\begin{abstract} %ここに概要を記�?
The Bergman fan of a matroid is the intersection of tropical hyperplanes defined by
the circuits. A tropical basis is a subset of the circuit set that defines the Bergman fan.
Yu and Yuster  posed a question whether every simple regular
matroid has a unique minimal tropical basis of its Bergman fan, and
verified it for graphic, cographic matroids and $R_{10}$.
We show every simple binary matroid has a unique minimal tropical basis.
Since the regular matroid is binary, we positively answered the question.
\end{abstract}

\section{Introduction} %セクションを作�?
Matroids were introduced by Whitney in 1935 to try to
capture abstractly the essence of dependence. A {\it matroid} is an ordered pair $(E,{\mathscr C})$ consisting of a finite set $E$ and a collection ${\mathscr C}$ of subset of $E$ having three properties:

\begin{enumerate}
 \setlength{\parskip}{0cm} % 段落��?
  \setlength{\itemsep}{0cm} % ��?����?
\renewcommand{\labelenumi}

  \item {\bf (1)} $\varnothing \notin {\mathscr C}$;
  \item {\bf (2)} if $C_{1}$ and $C_{2}$ are members of ${\mathscr C}$ and $C_{1} \subset C_{2}$, then $C_{1} =C_{2}$;
  \item {\bf (3)} if $C_{1}$ and $C_{2}$ are distinct members of ${\mathscr C}$ and $e \in C_{1}\cap C_{2}$,\\ \hspace{0,5cm} then there is a member $C_{3}$ of ${\mathscr C}$ such that $C_{3} \subset (C_{1}\cup C_{2})\,\backslash\,\{e\}$.
\end{enumerate}
${\mathscr C}$ is called the $circuit\ set$ and a member of ${\mathscr C}$ is called a $circuit$ of $M$. Two matroids $M=(E,{\mathscr C})$ and $M'=(E',{\mathscr C'})$ are isomorphic, in which case we write $M \cong M',$ if there is a bijection $\varphi : E \to E'$ such that for every subset $X\subset E$, $X$ is a circuit of $M$ if and only if $\varphi (X)$ is a circuit of $M'$. A matroid $M=(E,{\mathscr C})$ is {\it simple} if the cardinality of  every circuit of $M$ is greater than 2.\par
Here are some examples of matroids. Let $G$ be a finite graph,  $E$ be the edge set of  $G$, and ${\mathscr C}$ be the family of edge sets of  cycles in $G$. Then $(E,{\mathscr C})$ obeys the axioms for matroid. Any matroid that can be obtained in this way is called  a {\it cycle} matroid. If graph $G$ is simple, then the cycle matroid of $G$ is simple.
A matroid $M$ is said to be $graphic$ if there is a cycle matroid $M'$ such that  $M \cong M'$.
Let $K$ be a field, $V$ be a vector space over $K$,
$E$ be a finite subset of $V$, and ${\mathscr C}$ be the family of minimal linearly dependent subset of $E$.
Then $(E,\,{\mathscr C})$ obeys the axioms for matroid, that is a matroid.
Any matroid that can be obtained in this way is called a {\it linear } matroid over $K$.
A matroid $M$ is said to be {\it representable} over $K$ if there is a linear matroid $M'$ over $K$  such that $M \cong M'$. A matroid is called {\it binary } if it is representable over binary field $\mathbb F_{2}$, and {\it regular} if it is representable over every field.\par
The Bergman fan of a matroid is defined to be the intersection of tropical hyperplanes of tropical linear form of circuits. The Bergman of a matroid fan is a kind of tropical linear space, used as a local model of tropical manifold.
A subset of the circuit set $B$ is called a {\it tropical basis} if the intersection of tropical hyperplanes of tropical linear form of circuits in $B$ is equal to the Bergman fan.
A tropical basis is said to be {\it minimal} if it is minimal with respect to the inclusion relation.
For an ideal $I$ of a polynomial ring over a field, if $\{f_{1},\,f_{2},\,...,\,f_{m}\}$ is a generator of $I$, then $V(\{f_{1},\,f_{2},\,...,\,f_{m}\})$ is equal to $V(I)$. In the sense that the zero set of a tropical basis is equal to the Bergman fan of a matroid, a minimal tropical basis is an analogy to minimal generators of ideal. We study minimal tropical basis. In \cite{key7} Yu and Yuster showed that every graphic, cographic simple matroid and $R_{10}$ have a unique minimal tropical basis. The Seymour decomposition theorem states that every  regular matroid can be decomposed into those matroids by repeated 1-, 2-, and 3-sum decompositions. Yu and Yuster  posed a question whether every simple regular matroid has a unique minimal tropical basis. We proved that every simple binary matroid has a unique minimal tropical basis. Since the regular matroid is binary, we  positively answered the question. \par
In Section 2, we recall some basic facts for matroids and Bergman fan of matroid. Then in Section 3  we gave a necessary and sufficient  condition  for a matroid to have a unique minimal tropical basis.  We prove thet a simple uniform matroid is regular if and only if it has a unique minimal tropical basis.  For the Fano matroid and the so-called non-Fano matroid, we describe these minimal tropical bases. In Section 4 we show  our main result that if a simple matroid is binary, then it has a unique minimal tropical basis. Finally 
by using Yu and Yuster's Lemma, we give an algorithm discriminating whether a matroid has a unique minimal tropical basis. We give examples of matroids. $P_{7}$ and $R_{6}$ are two matroids. It is known that for a filed $K$, $P_{7}$ and $R_{6}$ are representable over $K$ if and only if the cardinality of $K$ is more than two.
By that algorithm, we know that $P_{7}$ has a unique minimal tropical basis but $R_{6}$ does not.
Therefore, for a filed $K$, every simple matroid that is representable over $K$ has a unique minimal tropical basis if and only if $K$ is binary filed.  
\section{Preliminaries}

We briefly recall the theory of matroid and Bergman fan of matroid. We refer the reader to \cite{key1,key2,key3} for
details and further references.
Let $M=(E,\,{\mathscr C})$ be a matroid and $T$ be a subset of $E$. We define ${\mathscr C_{M/T}}$ as the set that  consists of minimal nonempty elements of $\{C\,\backslash \,T\ |\ C\in {\mathscr C}\}$. $(E\,\backslash \,T,\, {\mathscr C_{M/T}})$ is a matroid. We call this matroid the $contraction$ of $T$ from $M$ and write it as $M/T$.
Let ${\mathscr C_{M\backslash T}}$ be the set $\{C\subset E\,\backslash \,T\ |\ C\in {\mathscr C}\}$. $(E\,\backslash
\,T, {\mathscr C_{M\backslash T}})$ is a matroid.
We call this matroid the $deletion$ of $T$ from $M$ and write it as $M\,\backslash \,T$. A {\it minor} of $M$ is any matroid that can be obtained from $M$ by a sequence of deletions or contractions. For $n$ and $d$ be natural numbers and $n\geq d$, a {\it uniform} matroid $U_{d,n}$ is defined as follows. The ground set is $[n] := \{1,2,...,n\}$. The circuit set of $U_{d,n}$ consists of every $(d+1)$-subsets of $[n]$. A uniform matroid $U_{d,n}$ is simple if and only if $d$ is more than one. A matroid is called {\it ternary} if it is representable over $\mathbb F_{3}$. 
A matroid $M$ is said to be {\it cographic} if there is a graphic matroid $M'$ such that the dual matroid of $M'$ is isomorphic to $M$. 
$R_{10}=(E,\,{\mathscr C})$ is a regular matroid. The cardinality of $E$ is ten. $R_{10}$ is neither graphic nor cographic.

There are characterizations for a matroid to be binary or regular.

\begin{theorem}[\cite{key3}]
Let $M$ be a matroid. The\ following\   statements\   are\   equivalent.
\begin{enumerate}
 \setlength{\parskip}{0cm} % 段落��?
  \setlength{\itemsep}{0cm} % ��?����?

\renewcommand{\labelenumi}

   \item {\bf (1)} $M$ is binary.
   \item {\bf (2)} $M$ has no minor isomorphic to $U_{2,4}.$
   \item {\bf (3)} If $C_{1}$ and $C_{2}$ are circuits, then their
symmetric difference $C_{1}\bigtriangleup C_{2}$ is a disjoint union of circuits.

\end{enumerate}
\end{theorem}

\begin{theorem}[\cite{key2}\cite{key6}]
Let $M$ be a matroid. The following statements are equivalent.
\begin{enumerate}
 \setlength{\parskip}{0cm} % 段落��?
  \setlength{\itemsep}{0cm} % ��?����?

\renewcommand{\labelenumi}

   \item {\bf (1)} $M$ is regular.
   \item {\bf (2)} $M$ is binary and ternary.
   \item {\bf (3)} $M$ can be decomposed into graphic and cographic matroids
and matroids isomorphic to $R_{10}$ by repeated 1-, 2-, and 3-sum decompositions.

\end{enumerate}
\end{theorem}

Let ${\bf T}$ = (${\bf R}$\,$\cup \,\{-\infty\}$,\,$\oplus$,\,$\odot$)  be the tropical semifield where $\oplus$ is maximum operation and $\odot$ is the usual addition. ${\bf TP}^{n}$ is the {\it tropical projective space}.
The {\it Bergman fan} or {\it constant coefficient tropical linear space} of a matroid can be described in terms of their  matroid.
Let $M=([n],\,{\mathscr C})$ be a matroid. For a circuit $C \in {\mathscr C}$, let $V(C)$ be the set of points $x\in$ ${\bf TP}^{n-1}$  such that the maximum value in $\{ x_{i}\ |
\ i \in C\} $ is attained at least twice. The set $V({\mathscr C}):= \bigcap_{C \in {\mathscr C}} V(C)$ is a polyhedral fan called the  {\it Bergman fan } or {\it constant\ coefficient\ tropical linear space} of $M$.
Federico Ardila and Caroline Klivans had shown that the Bergman fan of a matroid $M$ centered at the origin is geometric realization of the order complex of the lattice of flats of $M$.
The simplication of $M$ does not change lattice of flats, so we may assume that our matroid is simple.

\section{Tropical basis}
Let $M=([n],\,{\mathscr C})$ be a matroid.
A subset $B$ of ${\mathscr C}$ is defined to be a {\it tropical basis}\, if $V(B):=\bigcap_{C\in B} V(C)$ is equal to $V({\mathscr C})$. A tropical  basis is $minimal$ if every its proper subset is not a tropical basis. We study minimal tropical basis and whether a matroid has a unique minimal tropical basis.\par
If two circuit $C_{1} , C_{2}$ have a unique element in $C_{1} \cap C_{2}$,  $pasting$ them means taking 
their symmetric difference $C_{1}\bigtriangleup C_{2} = (C_{1} \backslash C_{2})\cup (C_{2} \backslash C_{1})$.
\begin{prop}[\cite{key7}] 
If $S$ $\subset$ ${\mathscr C}$ has the property that every other circuit of the matroid can  be obtained by successively 
pasting circuits in S, then S is a tropical basis.
\end{prop}

\subsection{Intersection of all tropical bases}
We denote the intersection of all tropical bases of $M$ by $B_{M}$. We show that $M$ has a unique minimal tropical basis if and only if $B_{M}$ is a tropical basis.

\begin{lemma}
Let $C$ be a circuit.
The\ following\   statements\   are\   equivalent.\begin{enumerate}
\setlength{\parskip}{0cm} % 段落��?
  \setlength{\itemsep}{0cm} % ��?����?

\renewcommand{\labelenumi}

   \item {\bf(1)} The circuit  $C\ is\ in\ B_{M}.$
   \item {\bf(2)} ${\mathscr C}\,\backslash \{C\}\ is\ not\ a\ tropical\ basis.$
   \item {\bf(3)} There is an element $ x \in $ ${\bf TP}^{n-1}$ such that x is not in $V(C)$ and for every other circuit $C',\, x$ is in $V(C')$.

\end{enumerate}
\end{lemma}
\begin{proof}By the definition of tropical basis, it is clear that 
{\bf (2)} is equivalent to {\bf (3)}. We show {\bf (1)} implies {\bf (2)}.
If $C$ is in $B_{M}$, then every tropical basis contains a circuit $C$. Therefore, ${\mathscr C}\,\backslash $$\{C\}$ is not a tropical basis.
We prove {\bf (3)} implies {\bf (1)}. Let $x$ be a point in ${\bf TP}^{n-1}$ that is not in $V(C)$ but for every other circuit $C'$, $V(C')$ contains $x$. If a subset of the circuit set $B$ does not contain $C$, $x$ is not in $V({\mathscr C})$,   but in $V(B)$. $B$ is not a tropical basis. All tropical bases contain a circuit $C$. Therefore, $C$ is in $B_{M}$.
\end{proof}
\begin{lemma}
The following statements are equivalent.
\begin{enumerate}
\setlength{\parskip}{0cm} % 段落��?
  \setlength{\itemsep}{0cm} % ��?����?

\renewcommand{\labelenumi}

   \item {\bf(1)} $M$ has a unique minimal tropical basis.
   \item {\bf(2)} $B_{M}$ is a tropical basis.
\end{enumerate}
\end{lemma}
\begin{proof}
We first prove {\bf (2)} implies {\bf (1)}. Let $B$ be the unique minimal tropical basis of $M$. By the uniqueness of minimal tropical basis and  finiteness of the cardinality of the circuit set, every tropical basis of $M$ contains $B$. Therefore, $B_{M} = B$ is a tropical basis.
We now show {\bf(1)} implies {\bf (2)}.  We show the contraposition. Let $B_{1}$ and $B_{2}$ be distinct minimal tropical bases of $M$. Then we have $B_{M} \subset B_{1}\cap B_{2} \subsetneq B_{1}$. By the minimality of $B_{1}$, $B_{M}$ is not a tropical basis of $M$.
\end{proof}
\subsection{Uniform matroid}
Yu and Yuster\cite{key7} found minimal tropical bases of uniform matroids. In this section we show that for a simple uniform matroid, it is regular if and only if it has a unique minimal tropical basis. 
\begin{prop}[\cite{key7}]
For every $i \in [n]$, $B_{i} := \{ C\ |\ i \in C\}$ is a minimal tropical basis of uniform matroid $U_{d,n}$.

\end{prop}
\begin{prop}
The following statements are equivalent. \begin{enumerate}
\setlength{\parskip}{0cm} % 段落��?
  \setlength{\itemsep}{0cm} % ��?����?

\renewcommand{\labelenumi}

\item {\bf(1)} A simple uniform matroid $U_{d,n}$ is regular.
   \item {\bf(2)} A simple uniform matroid $U_{d,n}$ is binary.
   \item {\bf(3)} A simple uniform matroid $U_{d,n}$ has a unique minimal tropical basis.

\end{enumerate}
\end{prop}
\begin{proof} It is clear that {\bf (1)} implies {\bf(2)} by definition.\\
We show {\bf (2)} implies {\bf (1)}. 
For $i \in [n]$,
we have \\$ U_{d,n} \backslash {\{i\}} \cong  \begin{cases}
    U_{d,n-1} &  \text{$d<n$}\\
    U_{d-1,n-1} &  \text{$d=n$},\\
 
  \end{cases} 
  $ \hspace{1cm}
      $    U_{d,n}/\{i\}\ \cong \begin{cases}
      U_{d-1,n-1} &  \text{$d>0$}\\
    U_{d,n-1} &  \text{$d=0$}.\\
      \end{cases} $
$\newline$ $\newline$
If $n$ is more then $d+1$, then the deletion of a $n-d+2$-elements set from $U_{d,n}$ is isomorphic to $U_{d,d+2}$. The contraction of a $d-2$-elements set form $U_{d,d+2}$ is isomorphic to $U_{2,4}$. $U_{d,n}$ has a minor isomorphic to $U_{2,4}$. By Theorem 1, $U_{d,n}$ is not binary. Therefore,
 if a uniform matroid $U_{d,n}$ is binary, then $n$ is equal to $d+1$ or $d$. It is clear that $U_{n,n},\,U_{n-1,n}$ are ternary.
By Theorem 2, $U_{n,n}$ and $U_{n-1,n}$ are regular. We next show {\bf (2)} implies {\bf (3)}.
Let $U_{d,n}$ be a simple and binary uniform matroid. Since $U_{d,\,n}$ is binary, $n$ is equal to $d$ or $d+1$.
The circuit set of $U_{n,n}$ and $U_{n-1,n}$ are 
$\varnothing$ and $\{ [n] \}$, respectively. They have a unique minimal tropical basis. We now show {\bf (3)} implies {\bf (2)}. We show the contraposition. If a uniform matroid $U_{d,n}$ is not binary, then $n$ is more than $d+1$.
By Proposition 2, $B_{1}$ and $B_{2}$ are minimal tropical bases of $U_{d,n}$. There are circuits $C_{1} , C_{2}$ such that 2 $\notin C_{1}$ and $1 \notin C_{2}$.  Since $C_{1}$ is not in $B_{2}$ and $C_{2}$ is not in $B_{1}$, $B_{1}$ is not equal to $B_{2}$.   Both are minimal tropical bases. Therefore, $U_{d,n}$ has two minimal tropical bases.
\end{proof}
\subsection{Fano and non-Fano matroid}

%WinTpicVersion4.30a
{\unitlength 0.1in%
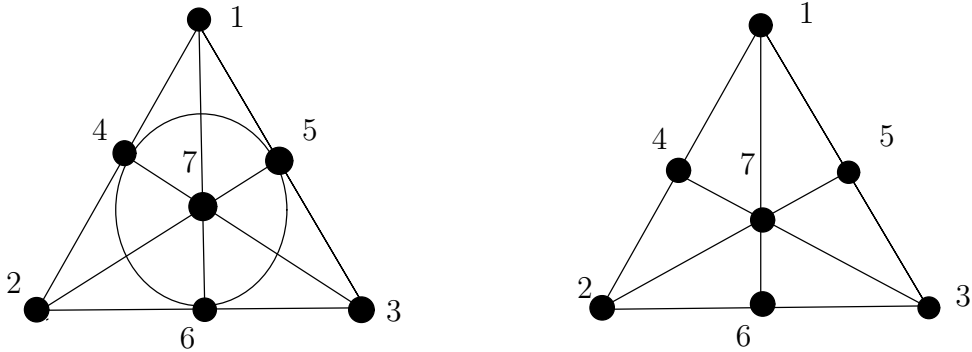
\begin{picture}( 49.7000, 17.0000)(  3.3000,-18.1000)%
% ELLIPSE 2 0 3 0 Black White
% 4 1352 1217 1800 1720 1647 1594 1647 1594
% 
\special{pn 8}%
\special{ar 1352 1217 448 503  0.0000000  6.2831853}%
% POLYGON 2 0 3 0 Black White
% 4 1340 250 2210 1730 493 1743 1340 250
% 
\special{pn 8}%
\special{pa 1340 250}%
\special{pa 2210 1730}%
\special{pa 493 1743}%
\special{pa 1340 250}%
\special{pa 2210 1730}%
\special{fp}%
% POLYGON 2 0 3 0 Black White
% 4 4280 250 5150 1720 3442 1738 4280 250
% 
\special{pn 8}%
\special{pa 4280 250}%
\special{pa 5150 1720}%
\special{pa 3442 1738}%
\special{pa 4280 250}%
\special{pa 5150 1720}%
\special{fp}%
% LINE 2 0 3 0 Black White
% 6 4280 240 4280 1730 3430 1740 4730 1020 5140 1720 3830 1020
% 
\special{pn 8}%
\special{pa 4280 240}%
\special{pa 4280 1730}%
\special{fp}%
\special{pa 3430 1740}%
\special{pa 4730 1020}%
\special{fp}%
\special{pa 5140 1720}%
\special{pa 3830 1020}%
\special{fp}%
% LINE 2 0 3 0 Black White
% 6 1340 260 1370 1730 500 1750 1750 960 2200 1740 950 930
% 
\special{pn 8}%
\special{pa 1340 260}%
\special{pa 1370 1730}%
\special{fp}%
\special{pa 500 1750}%
\special{pa 1750 960}%
\special{fp}%
\special{pa 2200 1740}%
\special{pa 950 930}%
\special{fp}%
% STR 2 0 3 0 Black White
% 4 1500 111 1500 260 2 0 0 0
% 1
\put(15.0000,-2.6000){\makebox(0,0)[lb]{1}}%
% STR 2 0 3 0 Black White
% 4 330 1550 330 1650 2 0 0 0
% 2
\put(3.3000,-16.5000){\makebox(0,0)[lb]{2}}%
% STR 2 0 3 0 Black White
% 4 2320 1700 2320 1800 2 0 0 0
% 3
\put(23.2000,-18.0000){\makebox(0,0)[lb]{3}}%
% STR 2 0 3 0 Black White
% 4 780 750 780 850 2 0 0 0
% 4
\put(7.8000,-8.5000){\makebox(0,0)[lb]{4}}%
% STR 2 0 3 0 Black White
% 4 1240 1840 1240 1940 2 0 0 0
% 6
\put(12.4000,-19.4000){\makebox(0,0)[lb]{6}}%
% STR 2 0 3 0 Black White
% 4 1880 750 1880 850 2 0 0 0
% 5
\put(18.8000,-8.5000){\makebox(0,0)[lb]{5}}%
% STR 2 0 3 0 Black White
% 4 1250 920 1250 1020 2 0 0 0
% 7
\put(12.5000,-10.2000){\makebox(0,0)[lb]{7}}%
% STR 2 0 3 0 Black White
% 4 4480 140 4480 240 2 0 0 0
% 1
\put(44.8000,-2.4000){\makebox(0,0)[lb]{1}}%
% STR 2 0 3 0 Black White
% 4 3320 1580 3320 1680 2 0 0 0
% 2
\put(33.2000,-16.8000){\makebox(0,0)[lb]{2}}%
% STR 2 0 3 0 Black White
% 4 5300 1620 5300 1720 2 0 0 0
% 3
\put(53.0000,-17.2000){\makebox(0,0)[lb]{3}}%
% STR 2 0 3 0 Black White
% 4 3710 800 3710 900 2 0 0 0
% 4
\put(37.1000,-9.0000){\makebox(0,0)[lb]{4}}%
% STR 2 0 3 0 Black White
% 4 4900 780 4900 880 2 0 0 0
% 5
\put(49.0000,-8.8000){\makebox(0,0)[lb]{5}}%
% STR 2 0 3 0 Black White
% 4 4150 1830 4150 1930 2 0 0 0
% 6
\put(41.5000,-19.3000){\makebox(0,0)[lb]{6}}%
% STR 2 0 3 0 Black White
% 4 4170 930 4170 1030 2 0 0 0
% 7
\put(41.7000,-10.3000){\makebox(0,0)[lb]{7}}%
% CIRCLE 2 0 0 0 Black Black
% 4 1340 220 1340 280 1340 280 1340 280
% 
\special{sh 1.000}%
\special{ia 1340 220 60 60  0.0000000  6.2831853}%
\special{pn 8}%
\special{ar 1340 220 60 60  0.0000000  6.2831853}%
% ELLIPSE 2 0 0 0 Black White
% 4 950 920 1010 985 890 920 890 920
% 
\special{sh 1.000}%
\special{ia 950 920 60 65  0.0000000  6.2831853}%
\special{pn 8}%
\special{ar 950 920 60 65  0.0000000  6.2831853}%
% CIRCLE 2 0 0 0 Black White
% 4 510 1730 580 1740 520 1750 860 2250
% 
\special{pn 8}%
\special{ar 510 1730 71 71  0.9783690  1.1071487}%
% CIRCLE 2 0 0 0 Black Black
% 4 490 1740 530 1790 530 1790 530 1790
% 
\special{sh 1.000}%
\special{ia 490 1740 64 64  0.0000000  6.2831853}%
\special{pn 8}%
\special{ar 490 1740 64 64  0.0000000  6.2831853}%
% CIRCLE 2 0 0 0 Black Black
% 4 1360 1200 1430 1220 1430 1220 1430 1220
% 
\special{sh 1.000}%
\special{ia 1360 1200 73 73  0.0000000  6.2831853}%
\special{pn 8}%
\special{ar 1360 1200 73 73  0.0000000  6.2831853}%
% CIRCLE 2 0 0 0 Black Black
% 4 1760 960 1770 1030 1770 1030 1770 1030
% 
\special{sh 1.000}%
\special{ia 1760 960 71 71  0.0000000  6.2831853}%
\special{pn 8}%
\special{ar 1760 960 71 71  0.0000000  6.2831853}%
% CIRCLE 2 0 0 0 Black Black
% 4 1370 1740 1430 1750 1430 1750 1430 1750
% 
\special{sh 1.000}%
\special{ia 1370 1740 61 61  0.0000000  6.2831853}%
\special{pn 8}%
\special{ar 1370 1740 61 61  0.0000000  6.2831853}%
% CIRCLE 2 0 0 0 Black Black
% 4 2190 1740 2250 1710 2250 1710 2250 1710
% 
\special{sh 1.000}%
\special{ia 2190 1740 67 67  0.0000000  6.2831853}%
\special{pn 8}%
\special{ar 2190 1740 67 67  0.0000000  6.2831853}%
% CIRCLE 2 0 0 0 Black Black
% 4 4280 250 4280 310 4280 310 4280 310
% 
\special{sh 1.000}%
\special{ia 4280 250 60 60  0.0000000  6.2831853}%
\special{pn 8}%
\special{ar 4280 250 60 60  0.0000000  6.2831853}%
% CIRCLE 2 0 0 0 Black Black
% 4 3850 1010 3870 1070 3870 1070 3870 1070
% 
\special{sh 1.000}%
\special{ia 3850 1010 63 63  0.0000000  6.2831853}%
\special{pn 8}%
\special{ar 3850 1010 63 63  0.0000000  6.2831853}%
% CIRCLE 2 0 0 0 Black Black
% 4 3450 1730 3490 1780 3480 1780 3480 1780
% 
\special{sh 1.000}%
\special{ia 3450 1730 64 64  0.0000000  6.2831853}%
\special{pn 8}%
\special{ar 3450 1730 64 64  0.0000000  6.2831853}%
% CIRCLE 2 0 0 0 Black Black
% 4 4290 1710 4340 1750 4340 1750 4340 1750
% 
\special{sh 1.000}%
\special{ia 4290 1710 64 64  0.0000000  6.2831853}%
\special{pn 8}%
\special{ar 4290 1710 64 64  0.0000000  6.2831853}%
% CIRCLE 2 0 0 0 Black Black
% 4 4290 1270 4350 1250 4350 1250 4350 1250
% 
\special{sh 1.000}%
\special{ia 4290 1270 63 63  0.0000000  6.2831853}%
\special{pn 8}%
\special{ar 4290 1270 63 63  0.0000000  6.2831853}%
% CIRCLE 2 0 0 0 Black Black
% 4 4740 1020 4770 1070 4770 1070 4770 1070
% 
\special{sh 1.000}%
\special{ia 4740 1020 58 58  0.0000000  6.2831853}%
\special{pn 8}%
\special{ar 4740 1020 58 58  0.0000000  6.2831853}%
% CIRCLE 2 0 0 0 Black Black
% 4 5160 1730 5200 1770 5200 1770 5200 1770“�??
% 
\special{sh 1.000}%
\special{ia 5160 1730 57 57  0.0000000  6.2831853}%
\special{pn 8}%
\special{ar 5160 1730 57 57  0.0000000  6.2831853}%
\end{picture}}
\newline \newline
Let $E$ be the set $\{1,\,2,\,...\,,\,7\}$ of points and let ${\mathscr C_{1}}$ and ${\mathscr C_{2}}$ be the collection of subset $X$ of $E$ in the left diagram and the right diagram, respectively such that $X$ is three colinear points or four points does not contain three colilnear points. 
$(E,\,{\mathscr C_{1}})$ and $(E,\,{\mathscr C_{2}})$ obey the axioms for matroid. $(E,\,{\mathscr C_{1}})$ and $(E,\,{\mathscr C_{2}})$
are called the ${\it Fano\ matroid}$ and the ${\it non}$-${\it Fano\ matroid}$, respectively. The above diagrams are called {\it geometric representations} of these matroids.
A circuit $X$ of these matroids is called a {\it line} if  the cardinality of $X$ is three.
Let $K$ be a field. The Fano matroid is representable over $K$ if and only if 
the characteristic of $K$ is two.
The non-Fano matroid is representable over $K$ if and only if the characteristic of $K$ is not two.
Therefore, the Fano matroid is not ternary but binary, and the non-Fano matroid is not binary but ternary.
Yu and Yuster\cite{key7} showed that the Fano matroid has a unique minimal tropical basis.
The unique minimal tropical basis of the Fano matroid is 7 lines.
All minimal tropical bases of the non-Fano matroid are \{6 lines, \{1,\,4,\,5,\,6\}\},\,\{6 lines, \{2,\,4,\,5,\,6\}\},\,\{6 lines, \{1,\,4,\,5,\,6\}\} and \{6 lines, \{4,\,5,\,6,\,7\}\}. 
Therefore, the Fano matroid has a unique minimal tropical basis and the non-Fano matroid does not.
\section{Binary matroid}
 Yu and Yuster have shown that every simple graphic, cographic matroid and $R_{10}$ have a unique minimal tropical basis. Concrete minimal tropical bases of these matroids are as follows.
 
 \begin{theorem}[\cite{key7}]
\hspace{5cm}
\begin{enumerate}
 \setlength{\parskip}{0cm} % 段落��?
  \setlength{\itemsep}{0cm} % ��?����?

\renewcommand{\labelenumi}

   \item {\bf (1)} The unique minimal tropical basis of a simple graphic matroid consists of
the induced cycles.
   \item {\bf (2)} The unique minimal tropical basis of a simple cographic matroid consists
of the edge cuts that split the graph into two 2-edge-connected
subgraphs.
   \item {\bf (3)} The unique minimal tropical basis of $R_{10}$ consists of the fifteen
4-cycles.

\end{enumerate}
\end{theorem}\noindent
In this section we show that every simple binary matroid has a unique minimal tropical basis. Let $M=([n],\,{\mathscr C})$ be a matroid. It is clear that  if  the cardinality of ${\mathscr C}$ is zero or one, then  $M$ has a unique minimal tropical basis.
So in this section we assume that the cardinality of ${\mathscr C}$ is more than one.

 \begin{prop}
 Let M be a simple binary matroid and $C$ be a circuit of $M$. The following statements are  equivalent.
 \begin{enumerate}
  \setlength{\parskip}{0cm} % 段落��?
  \setlength{\itemsep}{0cm} % ��?����?

\renewcommand{\labelenumi}

   \item {\bf (1)} The circuit C is not in $B_{M}$.
   \item {\bf (2)} There are circuits $C_{1},C_{2}$ such that $C_{1},C_{2}$ have a unique element in their\newline
\hspace{0.5cm} intersection and  their
symmetric difference $C_{1}\bigtriangleup C_{2}$ is equal to C.
 
 \end{enumerate}
\end{prop}

\begin{proof} By Proposition 1, it is clear that {\bf (1)} implies {\bf (2)}.\\
We prove that {\bf (2)} implies {\bf (1)}.  We show the contraposition. Let $C$ be a circuit that does not satisfy condition  {\bf (2)}.
We show that for every circuit $C'$ in ${\mathscr C}\,\backslash\{C\}$, the cardinality of $C'\backslash C$
is equal to or more than 2.  By simpleness of $M$,
if $C\cap C'$ is empty, then $|C'\backslash C|=|C'|$ is more than 2. So we assume that $C\cap C'$ is not empty. By axioms for matroid (2), $|C\backslash C'|$ and $|C'\backslash C|$ are more than or equal to one. We use contradiction for proof.
Assume that the cardinality of $C' \backslash C$ is equal to one. Let $e$ be a unique element of $C' \backslash C$. Since $M$ is binary, $C\bigtriangleup C'$ is a disjoint union of circuits $\coprod _{i=1}^{k} C_{i}$. There is a unique circuit $C_{i}$ that contains the element $e$.
$\coprod_{j=1\\,j\neq i}^{k} C_{j}$ does not contain the element $e$. 
We have $\coprod_{j=1\\j\neq i}^{k} C_{j} \subset C\bigtriangleup C' = (C\backslash C')\cup (C'\backslash C) = (C\backslash (C\cap C'))\cup \{e\}\subset C\cup \{e\}$.  Since $M$ is simple, the  cardinality of $C_{i}$ is more than 2. Therefore, $\coprod_{j=1\\j\neq i}^{k} C_{j}$ is a proper subset of $C$.
By the axioms for matroid (2), $k$ is equal to one. Hence
$C\bigtriangleup C'$ is a circuit $C_{1}$. We have $C'\cap C_{1} = C' \cap (C\bigtriangleup C') = C'\cap
 ((C\backslash C')\cup (C'\backslash C)) = (C'\cap (C\backslash C'))\cup (C'\cap (C\backslash C'))=\varnothing \cup \{e\}=\{e\}.$ Since $C'$ contains the element $e$, 
$C'\cap C_{1}$ is equal to $\{e\}$.
$C_{1}\bigtriangleup C' = (C_{1}\cup C')\backslash (C_{1}\cap C')=(((C\backslash C')\cup (C'\backslash C))\cup C')\backslash \{e\}=(C\cup C')\backslash \{e\}=(C\cup C')\backslash (C'\backslash C)=C$. Therefore, the circuit $C$ is the pasting of $C'$ and $C_{1}$, which contradicts the assumption.
 \par
$x=(x_{i})_{i}\in$ ${\bf TP}^{n-1}$ is defined following.
Assign all but one
point of $C$ weight 0. Assign weight 1 to the remaining point of $C$ and to all other
points of the ground set $E$. $x=(x_{i})_{i}$ is not in $V(C)$. Since for every circuit $C'$ in ${\mathscr C}\,\backslash\,\{C\}$, $|C'\backslash C| \geq 2$, $x=(x_{i})_{i}$ is in $V({\mathscr C}\backslash \{C\})$. By Lemma 1, $C$ is in the
intersection of all tropical bases $B_{M}$. 
\end{proof}
\begin{lemma}
Let M be a simple matroid and
B be the set consisting of circuits that can not be obtained by pasting.
 B is a tropical basis of M.
\end{lemma}
\begin{proof} By simpleness of $M$, if the pasting of $C_{1}$ and $C_{2}$ is a circuit $C$, then the cardinality of $C$ is more than the cardinality of $C_{1},C_{2}$.
By induction, we can check that $B$ satisfies Proposition 1's condition. Therefore, $B$ is a tropical basis of $M$.\end{proof}

\begin{theorem}
Every simple binary matroid has a unique minimal tropical basis.
\end{theorem}
\begin{proof} Proposition 4 and Lemma 3 show that the intersection of all tropical bases of every simple binary matroid is a tropical basis.
By Lemma 2, every simple binary matroid has a unique minimal tropical basis.
\end{proof}\noindent
Since the regular matroid is binary, every simple regular matroid has a unique minimal tropical basis.

\section{Algorithm discriminating whether a matroid has a unique minimal tropical basis}
We propose an algorithm discriminating whether a matroid has a unique minimal tropical basis.
In \cite{key7} it was shown that we can determine whether $B \subset {\mathscr C}$ is a tropical basis  by $0/1-$points in ${\bf TP}^{n-1}$. Therefore, the condition whether a  subset of the circuit set is a tropical basis is computable.
Let $M$ be a matroid and ${\mathscr C}$ be the circuit set and $k$  be the cardinality of ${\mathscr C}$.
By Lemma 1, the intersection of of tropical bases $B_{M}$ is computable.
We give an order ${\mathscr C} =\{C_{1},\,C_{2}\,, ...\,,\,C_{k}\}$. Let $B_{0}$ be the circuit set $C_{M}$.
For $C_{i}$, if $B_{i}\backslash \{C_{i}\}$ is a tropical basis, then $B_{i+1}$ is defined to be $B_{i}\backslash \{C_{i}\}$. Else $B_{i+1}$ is defined to be $B_{i}$. After repeating this operation for $k$ times, we get a minimal tropical tropical basis $B_{k}$. By Lemma 2, if $B_{k}$ is equal to $B_{M}$, $M$ has a unique minimal tropical basis. Else $M$ has more than one minimal tropical bases.\newline
%WinTpicVersion4.30a
{\unitlength 0.1in%
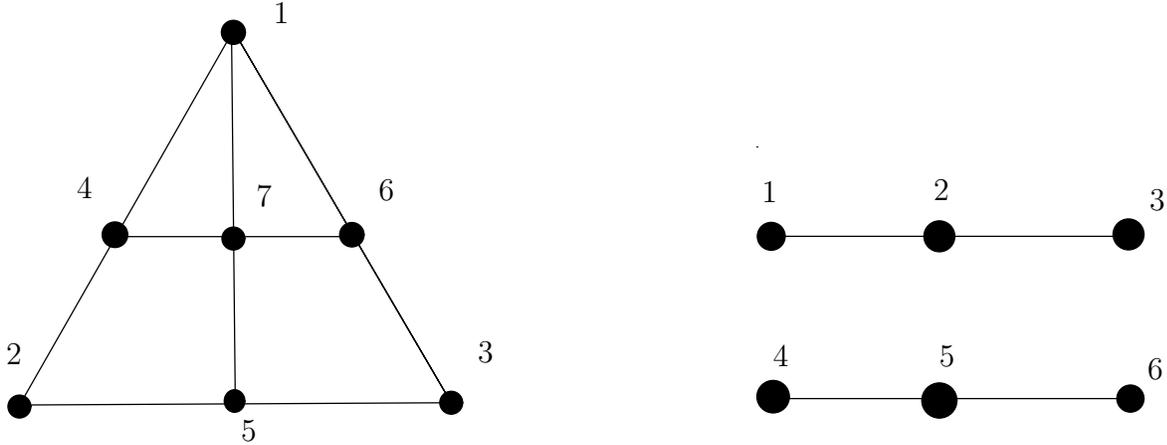
\begin{picture}( 39.3000, 22.0000)(  9.9000,-31.0000)%
% POLYGON 2 0 3 0 Black White
% 4 2180 1070 3320 3020 1061 3032 2180 1070
% 
\special{pn 8}%
\special{pa 2180 1070}%
\special{pa 3320 3020}%
\special{pa 1061 3032}%
\special{pa 2180 1070}%
\special{pa 3320 3020}%
\special{fp}%
% LINE 2 0 3 0 Black White
% 2 2170 1070 2190 3030
% 
\special{pn 8}%
\special{pa 2170 1070}%
\special{pa 2190 3030}%
\special{fp}%
% LINE 2 0 3 0 Black White
% 4 1570 2150 2820 2150 4920 1680 4920 1680
% 
\special{pn 8}%
\special{pa 1570 2150}%
\special{pa 2820 2150}%
\special{fp}%
\special{pa 4920 1680}%
\special{pa 4920 1680}%
\special{fp}%
% CIRCLE 2 0 0 0 Black Black
% 4 2180 1080 2160 1140 2160 1140 2160 1140
% 
\special{sh 1.000}%
\special{ia 2180 1080 63 63  0.0000000  6.2831853}%
\special{pn 8}%
\special{ar 2180 1080 63 63  0.0000000  6.2831853}%
% CIRCLE 2 0 0 0 Black Black
% 4 1560 2140 1624 2159 1624 2159 1624 2159
% 
\special{sh 1.000}%
\special{ia 1560 2140 67 67  0.0000000  6.2831853}%
\special{pn 8}%
\special{ar 1560 2140 67 67  0.0000000  6.2831853}%
% CIRCLE 2 0 0 0 Black Black
% 4 1060 3040 1120 3040 1120 3040 1120 3040
% 
\special{sh 1.000}%
\special{ia 1060 3040 60 60  0.0000000  6.2831853}%
\special{pn 8}%
\special{ar 1060 3040 60 60  0.0000000  6.2831853}%
% CIRCLE 2 0 0 0 Black Black
% 4 2800 2140 2840 2190 2840 2190 2840 2190
% 
\special{sh 1.000}%
\special{ia 2800 2140 64 64  0.0000000  6.2831853}%
\special{pn 8}%
\special{ar 2800 2140 64 64  0.0000000  6.2831853}%
% CIRCLE 2 0 0 0 Black Black
% 4 2180 2160 2240 2160 2240 2160 2240 2160
% 
\special{sh 1.000}%
\special{ia 2180 2160 60 60  0.0000000  6.2831853}%
\special{pn 8}%
\special{ar 2180 2160 60 60  0.0000000  6.2831853}%
% ELLIPSE 2 0 0 0 Black Black
% 4 2186 3011 2240 3070 2210 3064 2210 3064
% 
\special{sh 1.000}%
\special{ia 2186 3011 54 59  0.0000000  6.2831853}%
\special{pn 8}%
\special{ar 2186 3011 54 59  0.0000000  6.2831853}%
% CIRCLE 2 0 0 0 Black Black
% 4 3320 3020 3380 3020 3380 3020 3380 3020
% 
\special{sh 1.000}%
\special{ia 3320 3020 60 60  0.0000000  6.2831853}%
\special{pn 8}%
\special{ar 3320 3020 60 60  0.0000000  6.2831853}%
% STR 2 0 3 0 Black Black
% 4 2390 930 2390 1030 2 0 0 0
% 1
\put(23.9000,-10.3000){\makebox(0,0)[lb]{1}}%
% STR 2 0 3 0 Black Black
% 4 990 2720 990 2820 2 0 0 0
% 2
\put(9.9000,-28.2000){\makebox(0,0)[lb]{2}}%
% STR 2 0 3 0 Black Black
% 4 3460 2710 3460 2810 2 0 0 0
% 3
\put(34.6000,-28.1000){\makebox(0,0)[lb]{3}}%
% STR 2 0 3 0 Black Black
% 4 1360 1850 1360 1950 2 0 0 0
% 4
\put(13.6000,-19.5000){\makebox(0,0)[lb]{4}}%
% STR 2 0 3 0 Black Black
% 4 2220 3120 2220 3220 2 0 0 0
% 5
\put(22.2000,-32.2000){\makebox(0,0)[lb]{5}}%
% STR 2 0 3 0 Black Black
% 4 2940 1860 2940 1960 2 0 0 0
% 6
\put(29.4000,-19.6000){\makebox(0,0)[lb]{6}}%
% STR 2 0 3 0 Black Black
% 4 2300 1890 2300 1990 2 0 0 0
% 7
\put(23.0000,-19.9000){\makebox(0,0)[lb]{7}}%
\end{picture}}%
%WinTpicVersion4.30a
{\unitlength 0.1in%
\begin{picture}( 20.7500, 12.7200)(  5.0500,-13.5200)%
% LINE 2 0 3 0 Black White
% 4 590 400 2480 400 590 1250 2490 1250
% 
\special{pn 8}%
\special{pa 590 400}%
\special{pa 2480 400}%
\special{fp}%
\special{pa 590 1250}%
\special{pa 2490 1250}%
\special{fp}%
% CIRCLE 2 0 0 0 Black Black
% 4 580 400 650 420 650 420 650 420
% 
\special{sh 1.000}%
\special{ia 580 400 73 73  0.0000000  6.2831853}%
\special{pn 8}%
\special{ar 580 400 73 73  0.0000000  6.2831853}%
% CIRCLE 2 0 0 0 Black Black
% 4 1460 400 1470 480 1470 480 1470 480
% 
\special{sh 1.000}%
\special{ia 1460 400 81 81  0.0000000  6.2831853}%
\special{pn 8}%
\special{ar 1460 400 81 81  0.0000000  6.2831853}%
% CIRCLE 2 0 0 0 Black Black
% 4 2450 390 2490 460 2490 460 2490 460
% 
\special{sh 1.000}%
\special{ia 2450 390 81 81  0.0000000  6.2831853}%
\special{pn 8}%
\special{ar 2450 390 81 81  0.0000000  6.2831853}%
% CIRCLE 2 0 0 0 Black Black
% 4 590 1240 650 1300 650 1300 650 1300
% 
\special{sh 1.000}%
\special{ia 590 1240 85 85  0.0000000  6.2831853}%
\special{pn 8}%
\special{ar 590 1240 85 85  0.0000000  6.2831853}%
% CIRCLE 2 0 0 0 Black Black
% 4 1460 1260 1530 1320 1530 1320 1530 1320
% 
\special{sh 1.000}%
\special{ia 1460 1260 92 92  0.0000000  6.2831853}%
\special{pn 8}%
\special{ar 1460 1260 92 92  0.0000000  6.2831853}%
% CIRCLE 2 0 0 0 Black Black
% 4 2460 1250 2470 1320 2470 1320 2470 1320
% 
\special{sh 1.000}%
\special{ia 2460 1250 71 71  0.0000000  6.2831853}%
\special{pn 8}%
\special{ar 2460 1250 71 71  0.0000000  6.2831853}%
% STR 2 0 3 0 Black White
% 4 530 120 530 220 2 0 0 0
% 1
\put(5.3000,-2.2000){\makebox(0,0)[lb]{1}}%
% STR 2 0 3 0 Black White
% 4 1430 110 1430 210 2 0 0 0
% 2
\put(14.3000,-2.1000){\makebox(0,0)[lb]{2}}%
% STR 2 0 3 0 Black White
% 4 2560 160 2560 260 2 0 0 0
% 3
\put(25.6000,-2.6000){\makebox(0,0)[lb]{3}}%
% STR 2 0 3 0 Black White
% 4 590 980 590 1080 2 0 0 0
% 4
\put(5.9000,-10.8000){\makebox(0,0)[lb]{4}}%
% STR 2 0 3 0 Black White
% 4 1460 980 1460 1080 2 0 0 0
% 5
\put(14.6000,-10.8000){\makebox(0,0)[lb]{5}}%
% STR 2 0 3 0 Black White
% 4 2580 990 2580 1090 2 0 0 0
% 
\put(25.8000,-10.9000){\makebox(0,0)[lb]{}}%
% STR 2 0 3 0 Black White
% 4 2550 1050 2550 1150 2 0 0 0
% 6
\put(25.5000,-11.5000){\makebox(0,0)[lb]{6}}%
\end{picture}}%
\newline \newline 
We give examples.
Let $P_{7}$ be the matroid that has the left above diagram as geometric representation.
By that algorithm, we know $P_{7}$ has a unique minimal tropical basis. In \cite{key3} it was known that for a field $K$, $P_{7}$ is representable over $K$ if and only if the cardinality of $K$ is more than 2. So it is not binary. Therefore, the converse of Theorem 4 is not true.
Let $R_{6}$ be the matroid that has the right above diagram as geometric representation.
In \cite{key3} it was known that for a field $K$, $R_{6}$ is representable over $K$ if and only if the cardinality of $K$ is more than two. By that algorithm, we know $R_{6}$ has more than one minimal tropical bases.
By Theorem 4 and examples of $P_{7}$ and $R_{6}$, we have the following theorem.
\begin{theorem}
Let $F$ be a set of fields. The following statements are equivalent.
\begin{enumerate}
\setlength{\parskip}{0cm} % 段落��?
  \setlength{\itemsep}{0cm} % ��?����?

\renewcommand{\labelenumi}

\item {\bf(1)} If a simple matroid $M$ is representable over all fields in $F$, then $M$ has a unique minimal tropical basis.
\item {\bf(2)} The binary field ${\mathbb F}_{2}$ is in $F$.

\end{enumerate}

\end{theorem}

$E-mail\ address$: nakajima-yasuhito@ed.tmu.ac.jp
\end{document}